\documentclass[reqno]{amsart}
\usepackage{amsmath}
\usepackage[dvips]{graphicx}
\usepackage{amsfonts}
\usepackage{amssymb}
\usepackage{latexsym}
\graphicspath{{Img/}}

\newtheorem{theorem}{Theorem}

\theoremstyle{plain}

\newtheorem{corollary}{Corollary}

\newtheorem{lemma}{Lemma}

\newtheorem{problem}{Problem}

\numberwithin{equation}{section}

\begin{document}
\title[A generalization of Pythagoras on a aurface]{A generalization of Pythagoras on a surface}
\author{Anastasios Zachos}

\address{Greek Ministry of Education, Athens, Greece}
 \email{azachos@gmail.com}
 \keywords{Pythagorean theorem, law of sines , law of cosines,  geodesic triangles, surface} \subjclass{51E10,
52A10, 52A41, 53C45, 53C22.}
\begin{abstract}
We analyze Toponogov's sine theorem for an infinitesimal geodesic triangle $\triangle ABC$ on a $C^{2}$ regular surface $M,$ which is given in his book \cite[Problem~3.7.2]{VToponogov:05} and we provide a generalization of the law of cosines for $\triangle ABC$ on $M.$ By replacing in the law of cosines $\angle B=\frac{\pi}{2}$ on $M$, we derive the generalized theorem of Pythagoras on a surface:

\[AC^2=AB^2+BC^2+f(\angle A, \frac{\pi}{2}, AB, BC) o(AC^{2})\]

or

\[AC^2=AB^2+BC^2+(\angle A+ \angle C-\frac{\pi}{2})^2.\]

where $f(\angle A, \angle B, AB, BC)$ is a rational function w.r. to $\cos\angle A,$ $\cos\angle B,$ $\sin\angle A,$ $\sin \angle B,$ $AB$ and $BC.$

\end{abstract}\maketitle

\section{Introduction}
The law of cosines introduced by Euclid in his Elements (Book~II, Proposition~12, 13 in \cite{EuclidHeath:56} ), without using the term cosine, for obtuse angled and acute angled triangles in the Euclidean plane $\mathbb{R}^{2}.$

In triangle $\triangle ABC$ if angle is obtuse or acute, then
\[AC^{2}=AB^2+BC^2-2AB BC \cos\angle B, \]
\[BC^{2}=AB^2+AC^2-2AB AC \cos\angle A, \]
\[AB^{2}=AC^2+BC^2-2AC BC \cos\angle C. \]

This is the law of cosines (Cosine Theorem) in $\triangle ABC.$
By setting, for instance $\angle B=\frac{\pi}{2},$ we derive the Pythagorean theorem

\[AC^{2}=AB^2+BC^2. \]

A generalization of the law of cosines and the law of sines for the two dimensional sphere $S^{2}$ and the hyperbolic plane $H^2$ is given by W. Thurston in his book \cite[Chapter~2.4]{WThurston:}, by using vector calculus. W.Thurston considered three unit vectors lying in two dimensional sphere $S^{2}$ or to a Lorenz Space $\mathbb{R}_{1}^{2}$  in $\mathbb{R}^{3}$ and defined the dual basis to these vectors and their dot products and applied an inversion w.r to a $3\times 3$ matrix. This vector process yields the spherical and hyperbolic law of cosines and sines:

Cosine law for a geodesic triangle (arcs of great circles on the sphere) $\triangle ABC$ on the unit sphere $S^{2}(1)$
\[\cos AC=\cos AB \cos BC+\sin AB \sin BC \cos\angle B,\]
\[\cos BC=\cos AB \cos AC+\sin AB \sin AC \cos\angle A,\]
\[\cos AB=\cos AC \cos BC+\sin AC \sin BC \cos\angle C,\]

Sine law for a geodesic triangle $\triangle ABC$ on the unit sphere $S^{2}(1)$

\[\frac{\sin AC}{\sin\angle B}=\frac{\sin BC}{\sin\angle A}=\frac{\sin AB}{\sin\angle C}\]

Cosine law for a geodesic triangle $\triangle ABC$ on the hyperbolic plane $H^{2}(1)$
\[\cosh AC=\cosh AB \cosh BC-\sinh AB \sinh BC \cos\angle B,\]
\[\cosh BC=\cosh AB \cosh AC-\sinh AB \sinh AC \cos\angle A,\]
\[\cosh AB=\cosh AC \cosh BC-\sinh AC \sinh BC \cos\angle C,\]

Sine law for a geodesic triangle $\triangle ABC$ on the hyperbolic plane $H^{2}(1)$

\[\frac{\sinh AC}{\sin\angle B}=\frac{\sinh BC}{\sin\angle A}=\frac{\sinh AB}{\sin\angle C}.\]

Berg and Nikolaev derived a unified cosine law for the $K$-plane (a sphere with constant Gaussian curvature $K>0$ $S^{2}_{K}$ and a hyperbolic plane with constant Gaussian curvature $-K<0$ $H^{2}_{K}$).

We denote by

\begin{displaymath}
\kappa = \left\{ \begin{array}{ll}
\sqrt{K} & \textrm{if $K>0$,}\\
i\sqrt{-K} & \textrm{if $K<0$.}\\
\end{array} \right.
\end{displaymath}

The unified cosine law for $\triangle ABC$ is given by:
\begin{equation} \label{eqvar1}
\cos(\kappa PQ)=\cos(\kappa PR)\cos(\kappa
RQ)+\sin(\kappa PR )\sin(\kappa RQ)\cos(\angle R),
\end{equation}

for $R\in \{A,B,C\}.$

By replacing $\angle R=\frac{\pi}{2}$ in (\ref{eqvar1}), we obtain the Pythagorean theorem on surfaces with constant
Gaussian curvature ($S^{2}_{K},$ $H^{2}_{K}$).
\[\cos(\kappa PQ)=\cos(\kappa PR)\cos(\kappa
RQ)\]

If the radius $R\to +\infty$ of the sphere $S_{K}^{2},$ then we derive the law of cosines in $\mathbb{R}^{2}.$

G. Darboux succeeded in deriving the law of cosines for an infinitesimal geodesic triangle $\triangle ABC$ on a smooth surface (\cite[Chapter~VII]{Cartan:},\cite[Livre~VI, Chapter~VIII]{Darboux:93}).
By introducing normal coordinates w.r. to a vertex, for instance at $B$ and by assuming that
$AB=AB_{0}$ and $BC=BC_{0},$ ($AB_{0}$, $BC_{0}$ are planar linear segments), the following generalization of
the law of cosines is given:

\[AC^{2}=AB^2+BC^2-2AB BC \cos\angle B -\frac{1}{3}Kh^2 AC^2\]

or

\[AC^{2}=AB^2+BC^2-2AB BC \cos\angle B -\frac{2}{3}K S AB BC \sin\angle B\]

or

\[AC^{2}=AB^2+BC^2-2AB BC \cos(\angle B -\frac{KS}{3}),\]

where

\[KS=\angle A+\angle B+\angle C-\pi,\]

where $h$ is the height of triangle $\triangle ABC$ from $B$
and $K$ is the Gaussian curvature of the space in the direction of the planar element of the triangle.

By replacing $\angle B=\frac{\pi}{2},$ we obtain:

\[AC^{2}=AB^2+BC^2-\frac{2}{3}AB BC (\angle A +\angle C-\frac{\pi}{2}).\]
If we consider as an infinitesimal number $o(AC)=(\angle A +\angle C-\frac{\pi}{2})<<1,$
then multiplied by $-\frac{2}{3}AB BC$ yields
\[AC^{2}=AB^2+BC^2+ (\angle A +\angle C-\frac{\pi}{2}).\]
This formula may be considered as a generalization of Pythagoras on a surface for infinitesimal right geodesic triangles in the sense of Darboux.
Thus, we consider the following problem.

\begin{problem}
Can we derive a generalization of Pythagoras for bigger infinitesimal geodesic triangles than the ones introduced by Darboux  on a $C^{2}$ regular surface $M$ having geodesics without self intersections?
\end{problem}

V. Toponogov introduced an important generalization for infinitesimal geodesic triangles on $M,$ which are bigger than the infinitesimal geodesic triangles in the sense of Darboux.
We obtain a positive answer w.r to Problem~1, by using Toponogov's sine theorem for infinitesimal geodesic triangles on $M$ and by generalizing the law of cosines for infinitesimal geodesic triangles on $M.$ As a special case, we derive the theorem of Pythagoras for right infinitesimal geodesic triangles on a $C^{2}$ regular surface $M.$

\section{Understanding Toponogov's law of sines on a surface}
We denote by $\triangle{ABC}$ an infinitesimal geodesic triangle on a regular surface of class $C^{2}$ in $\mathbb{R}^{3},$ by $AB, BC, AC$ the length of the infinitesimal geodesic arcs,
$AC=\delta$ and we set $\angle A\equiv \alpha,$ $\angle B\equiv \beta,$ and $\angle C\equiv \gamma.$

We continue by analyzing the proof given by V. Toponogov in \cite[Problem~3.7.2, Solution]{VToponogov:05}.

We note that an arc length parameterization $c(s)$ counting from the vertex $A$ to $C$ shall be used.
Let $\sigma (s)$ be a geodesic through $c(s),$ such that $\gamma\equiv \angle (\sigma(s), AC)$ and $B(s)=\sigma(s)\cap AB,$ $t(s)=AB(s),$ $l(s)=A(s)B(s)$ and $\beta(s)=\angle AB(s)C.$

By taking into account that $A(s)=c(s)$ and by applying the first variational formula of the length of geodesics. w.r to the arc length $s,$ yields:

\[\frac{dl}{ds}=\cos(\gamma)+\cos(\beta(s))\frac{dt}{ds}\]

We note that the physical parameter $s$ corresponds to the parametrization on $AC$
and not on $AB.$
Therefore, in the first variational formula of the length of geodesics
(\cite[Lemma 3.5.1]{VToponogov:05}), we need to set $\cos(\beta(s))\frac{dt}{ds}$, instead
of $\cos(\beta(s)).$
Then, he uses the following lemma:

\begin{lemma}\label{lem1}
Take on $M$ an infinitesimal geodesic triangle $\triangle ABC$ and assume that a region $D$ bounded by $\triangle ABC$ is homeomorphic to a disk. The following two formulas connects the angles of $\triangle ABC$ with the Gaussian curvature $K$ the region $D$ and the Landau symbol $o(AC):$
\begin{equation}\label{gb}
\angle A+ \angle B +\angle C-\pi=\int\int_{D}KdS,
\end{equation}

where

\begin{equation}\label{gb2}
\int\int_{D}KdS=o(s).
\end{equation}

\end{lemma}
Lemma~1 is a special case of a classical theorem of Gauss-Bonnet for an infinitesimal region $D.$
 By applying Lemma~1 w.r. to $\triangle AB(s)C(s),$ we get:
\begin{equation}\label{eq1nimp}
\beta(s)= \pi-\alpha-\gamma+o(s)
\end{equation}
\[\cos(\beta(s))=\cos(\pi-\alpha-\gamma+o(s))\]
\[\cos(\beta(s))=-\cos(\alpha+\gamma) +o(s)
\sin(\alpha+\gamma)+o(o(s)).\] This is an expansion of Taylor
series with respect to $(\pi-\alpha-\gamma)$! and we get:
\[\cos(\beta(s))=-\cos(\alpha+\gamma) +o(s) \sin(\alpha+\gamma)+o(s)\]
\begin{equation}\label{eq2nimp}
\frac{dl}{ds}=\cos(\gamma)-\cos(\alpha+\gamma)\frac{dt}{ds}+o(s)\sin(\alpha+\gamma)\frac{dt}{ds}+o(s)\frac{dt}{ds}
\end{equation}

The function $t(s)$ is not supposed to be linear
function w.r. $s.$ Thus, $dt/ds$ does not equal to a
constant number, but it is continuous on the interval
$[0,\delta]$, therefore it is bounded on $[0,\delta]$ by a
constant number $C_{1,AB}$. Integrating both parts of the
inequality $(\int_{0}^{\delta}{dt/ds}ds)\le
C_{1,AB}\int_{0}^{\delta}{ds})$ and we get $AB\le C_{1,AB}AC$.

Therefore, by integrating (\ref{eq2nimp}) w.r. to $s$ from $0$ to
$\delta$, yields:
\[\l(\delta)=BC=\cos(\gamma)\delta-\cos(\alpha+\gamma)AB+(\sin(\alpha+\gamma)+1)o(\delta)\]
($C_{1,AB}$=$AB/AC$ after integrating with respect to s)

or

\begin{equation}\label{BC2}
BC+\cos(\alpha+\gamma)AB=\cos(\gamma)\delta+(\sin(\alpha+\gamma)+1)o(\delta)
\end{equation}

Similarly, we derive that:
\[AB=\cos(\alpha)\delta-\cos(\alpha+\gamma)BC+(\sin(\alpha+\gamma)+1)o(\delta)\]

or

\begin{equation}\label{AB2}
\cos(\alpha+\gamma)BC+AB=\cos(\alpha)\delta+(\sin(\alpha+\gamma)+1)o(\delta)
\end{equation}

The solution of the linear system of (\ref{BC2}), (\ref{AB2}) w.r. to $AB,$ $BC$
yields V. Toponogov's sine theorem for infinitesimal geodesic triangles on a surface:


\begin{theorem}[The Sine Theorem of Toponogov]\cite[Problem and Solution~3.7.2]{VToponogov:05}
\begin{equation}\label{eq1}
AB=\frac{\sin(\gamma)\delta}{\sin(\alpha+\gamma)}+\frac{o(\delta)(1+\sin(\alpha+\gamma))}{(1+\cos(\alpha+\gamma)}
\end{equation}
and
\begin{equation}\label{eq2}
BC=\frac{\sin(\alpha)\delta}{\sin(\alpha+\gamma)}+\frac{o(\delta)(1+\sin(\alpha+\gamma))}{1+\cos(\alpha+\gamma)}.
\end{equation}
\end{theorem}

\section{The cosine theorem on a surface}

By using Toponogov's theorem for an infinitesimal geodesic triangle $\triangle ABC$ on $M,$
we obtain a generalization of the law of cosines for infinitesimal geodesic triangles on a surface $M.$

\begin{theorem}[Cosine Theorem]\label{cos1}
The law of cosines of an infinitesimal geodesic triangle on $M$ is given by:
\begin{equation}\label{sur1}
AC^2=AB^2+BC^2-2 AB BC \cos(\beta)+f(\alpha,\beta,AB,BC) o(AC^2)
\end{equation}
where $f(\alpha,\beta,AB,BC)$ is a rational function w.r. to $\cos\alpha,$ $\sin(\alpha),$ $\cos\beta,$ $\sin\beta,$
$AB$ and $BC$

or
\begin{equation}\label{gcosss2}
AC^2=AB^2+BC^2-2 AB BC \cos(\beta)+f(\alpha,\beta,AB,BC) (\angle A+ \angle B+\angle C-\pi)^2.
\end{equation}

\end{theorem}

\begin{proof}
We set $m(\alpha,\gamma)\equiv 1+\sin(\alpha+\gamma).$

From (\ref{eq2}) of the sine theorem of Toponogov, we get:
\begin{equation}\label{gcos1}
\sin(\alpha+\gamma) BC = \sin(\alpha) \delta + \frac{o(\delta)\sin(\alpha+\gamma)m(\alpha,\gamma)}{m(\frac{\pi}{2}-\alpha,\frac{\pi}{2}-\gamma)}
\end{equation}
By setting $w(\alpha,\gamma)\equiv \frac{\sin(\alpha+\gamma)m(\alpha,\gamma)}{m(\frac{\pi}{2}-\alpha,\frac{\pi}{2}-\gamma)}$
and by replacing $w(\alpha,\gamma)$ in (\ref{gcos1}), we derive that:

\begin{equation}\label{gcos2}
\sin(\alpha+\gamma) BC = \sin(\alpha) \delta + w(\alpha,\gamma) o(\delta)
\end{equation}

By squaring both parts of (\ref{BC2}) and (\ref{gcos2}) and by adding the two derived equations, we obtain that:

\begin{eqnarray}\label{gcos3}
AB^2+BC^2+2 AB BC \cos(\alpha+\gamma)=AC^2+ (m^2(\alpha,\gamma)+w^2(\alpha,\gamma))o(AC^2)+\\\nonumber
+2 (\cos\alpha m(\alpha,\gamma)+\sin\alpha w(\alpha,\gamma))AC o(AC)
\end{eqnarray}

By using the property of Landau symbol $o(AC^2)=AC o(AC),$ we obtain:

\begin{eqnarray}\label{gcos4}
AB^2+BC^2+2 AB BC \cos(\alpha+\gamma)=AC^2+\\\nonumber+
(m^2(\alpha,\gamma)+w^2(\alpha,\gamma)+2 (\cos\alpha m(\alpha,\gamma)+\sin\alpha w(\alpha,\gamma))o(AC^2).
\end{eqnarray}

By replacing $\alpha+\gamma=\pi-\beta+o(AC)$ in $\cos(\alpha+\gamma),$ we obtain:

\begin{eqnarray}\label{cosag}
\cos(\alpha+\gamma)=\cos(\pi-\beta+o(AC))=\cos\beta(1-o(AC)^2/2)-\sin\beta (o(AC))
\end{eqnarray}
By replacing (\ref{cosag}) in (\ref{gcos4}) and taking into account properties of the Landau symbols $o,$
we obtain:

\begin{equation}\label{gcosss}
AC^2=AB^2+BC^2-2 AB BC \cos(\beta)+f(\alpha,\beta, AB, BC) (\angle A+ \angle B+\angle C-\pi)^2.
\end{equation}

\end{proof}

\begin{corollary}
For $\angle A+\angle B+\angle C=\pi,$ we derive the law of cosines in $\mathbb{R}^2.$
\end{corollary}

By replacing $\angle B=\frac{\pi}{2}$ in Theorem~~2, we derive as a special case of the cosine theorem for infinitesimal geodesic triangles on $M$ the generalized Pythagorean Theorem for right infinitesimal geodesic triangles  on a surface $M.$

\begin{theorem}[The generalized Pythagorean Theorem on $M$]

The generalized theorem of Pythagoras for $\triangle ABC,$ and $\angle B=\frac{\pi}{2}$ on a surface
is given by:
\[AC^2=AB^2+BC^2+f(\angle A, \frac{\pi}{2}, AB, BC) o(AC^{2})\]

or

\[AC^2=AB^2+BC^2+(\angle A+ \angle C-\frac{\pi}{2})^2.\]
\end{theorem}

\end{document}